\documentclass[a4paper,12pt,oneside]{amsart}

\usepackage{amssymb}  
\usepackage{latexsym} 
\usepackage{comment}
\usepackage{color}
\usepackage{colonequals}

 \usepackage{epsfig}
\usepackage{tikz}

\definecolor{darkgreen}{rgb}{0,0.5,0}
\usepackage[
        colorlinks, citecolor=darkgreen,
        pdfauthor={Ingrid Bauer, Roberto Pignatelli},
        pdftitle={Rigid, non infinitesimally rigid surfaces},
        linktocpage        
]{hyperref}

\usepackage[alphabetic,backrefs,lite]{amsrefs} 

\usepackage[all]{xy} 

\usepackage{enumerate} 

\newcommand\sB{{\mathcal B}}

\DeclareMathOperator{\Pic}{Pic}

\DeclareMathOperator{\Def}{Def}

\DeclareMathOperator{\kod}{kod}

\newcommand{\CC}{\ensuremath{\mathbb{C}}}

\newcommand{\ZZ}{\ensuremath{\mathbb{Z}}}

\newcommand{\NN}{\ensuremath{\mathbb{N}}}
\newcommand{\hol}{\ensuremath{\mathcal{O}}}

\newcommand{\PP}{\ensuremath{\mathbb{P}}}

\newcommand{\ra}{\ensuremath{\rightarrow}}

\def\eea{\end{eqnarray*}}
\def\bea{\begin{eqnarray*}}

\newcommand\dual{\mathrel{\raise3pt\hbox{$\underline{\mathrm{\thinspace d
\thinspace}}$}}}
\newcommand\qe{\ifhmode\unskip\nobreak\fi\quad $\Box$}       

\def\BOX{\hfill\lower.5\baselineskip\hbox{$\Box$}}

\newtheorem{theorem}{Theorem}[section]
\newtheorem{lemma}[theorem]{Lemma}
\newtheorem{corollary}[theorem]{Corollary}
\newtheorem{proposition}[theorem]{Proposition}

\newtheorem*{theorem*}{Theorem}
\newtheorem*{problem*}{Problem}
\newtheorem*{question*}{Question}

\theoremstyle{remark}
\newtheorem{remark}[theorem]{Remark}

\theoremstyle{definition}
\newtheorem{definition}[theorem]{Definition}
\newtheorem{notation}[theorem]{Notation}

\newcommand{\cl}{\ensuremath{\mathcal{L}}}
\newcommand{\oo}{\ensuremath{\mathcal{O}}}

\numberwithin{equation}{section}

\newcounter{nootje}
\renewcommand\check[1]
  {\marginpar{\tiny\begin{minipage}{20mm}\begin{flushleft}\thenootje : #1\end{flushleft}\end{minipage}}\addtocounter{nootje}{1}}
\begin{document}

\title[Rigid but not infinitesimally rigid]{Rigid but not infinitesimally rigid compact complex manifolds}

\author{Ingrid Bauer}
\address{Mathematisches Institut,
         Universit\"at Bayreuth,
         95440 Bayreuth, Germany.}
\email{Ingrid.Bauer@uni-bayreuth.de}

\author{Roberto Pignatelli}
\address{Dipartimento di Matematica,
	Universit\`a di Trento,
	via Sommarive 14,
	I-38123 Trento, Italy.}
\email{Roberto.Pignatelli@unitn.it}
\date{\today}
\thanks{
Both authors thank Fabrizio Catanese for many useful mathematical discussions, the second author is grateful to Fabrizio Catanese for inviting him to Bayreuth with the ERC - 2013 - Advanced Grant - 340258 - TADMICAMT; this research took place during his visit.
He  is  a  member  of  GNSAGA - INdAM and  was partially  supported  by  the  project  PRIN 2015 Geometria delle 
variet\`a algebriche.\\
Both authors are grateful to Davide Frapporti for  figures \ref{fig:omega1} and \ref{fig:omega2}.\\
The authors are indebted to the referees for several suggestions improving the presentation of the paper, and especially for the proof of Lemma \ref{wellknown}, which is much simpler than the original one}

 \makeatletter
    \def\@pnumwidth{2em}
  \makeatother

\begin{abstract}
The aim of this paper is to  give for each dimension $d \geq 2$  an infinite series of rigid compact complex manifolds which are not infinitesimally rigid, hence to give an exhaustive answer to a problem of Morrow and Kodaira stated in the famous book {\em Complex manifolds}.
\end{abstract}

\maketitle

\addtocontents{toc}{\protect\setcounter{tocdepth}{1}}

\tableofcontents

\section*{Introduction}

In the famous book {\em Complex manifolds} by J. Morrow and K. Kodaira the following problem is posed

\begin{problem*}\cite{kodairamorrow}*{p. 45}
Find an example of a (compact complex manifold) $M$ which is rigid, but $H^1(M, \Theta) \neq 0$. (Not easy?)
\end{problem*}

A  compact complex manifold is {\em rigid} if it has no nontrivial (small) deformations (cf. Definition \ref{rigid}). Moreover, recall that a complex manifold $M$ is called {\em infinitesimally rigid}, if 
$H^1(M, \Theta_M) = 0$, and that, by Kuranishi theory, infinitesimal rigidity implies rigidity (cf. \cite{kodairamorrow}*{Theorem 3.2}). 

The above Problem asks for examples of compact complex manifolds which are rigid, but not infinitesimally rigid, showing that the converse of \cite{kodairamorrow}*{Theorem 3.2} does not hold.

To our knowledge this problem is up to now unsolved and the aim of our article is to give an infinite series of such examples for each dimension $d \geq 2$.

In \cite{rigidity} several different notions of rigidity (cf. Definition \ref{rigid}, where we repeat the notions which are relevant for our purposes) have been recalled and newly introduced and their relations have been studied. 

It is wellknown that in dimension 1 all concepts of rigidity coincide and that the only rigid curve is $\PP^1$.

In dimension 2, in \cite{rigidity} the following was proven  (in the slightly more general context of compact complex surfaces):

\begin{theorem*}
Let $S$ be a smooth projective  surface, which is  rigid. Then either
\begin{enumerate}
\item $S$ is a minimal surface of general type, or
\item $S$ is a Del Pezzo surface of degree $d \geq 5$.
\end{enumerate}
\end{theorem*}

Del Pezzo surfaces  are  infinitesimally rigid, and rigid surfaces of general type are also globally rigid due to the existence of a moduli space.

The above result seems to suggest that the  property of rigidity puts strong restrictions on the Kodaira dimension of the manifold $X$, but if we go to higher dimensions this is no longer true. In fact, in \cite{rigidity} the following is shown:

\begin{theorem*}
For each $n \geq 3$ and for each $k = - \infty, 0,2, \dots n$ there is a rigid projective variety $X$ of dimension $n$ and Kodaira dimension $\kod (X) =k$. 
\end{theorem*}

The above stated  result on rigid surfaces shows therefore that the problem of classifying rigid  surfaces reduces to the same question for surfaces of general type, and the list of known rigid surfaces of general type is rather short. Again we refer to \cite{rigidity} for a detailed account of the status of the art.

Among the several questions raised in \cite{rigidity} there is the following \cite{rigidity}*{Question 1.5. B}, special case of the problem of Morrow and Kodaira:
\begin{question*}
Does there exist a rigid, but not infinitesimally rigid surface of general type?
\end{question*}
This means that the moduli space of such surfaces consists of a single non reduced point.  

On one hand, the existence of such surfaces is expected in view of Murphy's law for moduli spaces  (cf. \cite{murphy}), which says that however bad a singularity is it appears as singular locus of some moduli space and since there are known examples of everywhere non reduced moduli spaces (cf. \cite{Ca89}). Still the proofs of these results rely on constructions where the moduli spaces have to be  positive dimensional.

On the other hand, showing rigidity can be quite difficult, and usually there are only techniques which allow to show the rigidity of a surface of general type proving  the vanishing of $H^1(\Theta_S)$. This probably is the reason that the problem of Morrow and  Kodaira  remained open for more than 45 years.

In this paper we give an infinite series of rigid regular surfaces of general type with unbounded invariants ($p_g$, $K^2$). More precisely, our first main result is:
\begin{theorem*}
For every even $n \geq 8$ such that $3 \nmid n$ there is a minimal regular surface $S_n$ of general type with 
$$
K^2_{S_n}=2(n-3)^2, \ \  p_g({S_n}) =\left( \frac{n}2-2 \right)\left( \frac{n}2-1 \right),
$$
such that $S_n$ is rigid, but not infinitesimally rigid.
\end{theorem*}

This gives a positive answer to the question of Morrow and Kodaira in dimension $2$.

The surface $S_n$ is constructed as the minimal resolution of singularities of a so-called {\em product-quotient surface} whose singular model $X_n$ has six nodes. More precisely,  $X_n$ is the quotient of a product of two algebraic curves $C_1 \times C_2$ by the faithful action of a finite group $G$, such that $G$ acts on each factor and the quotient map $C_i \ra C_i/G\cong \PP^1$ is branched in three points (i.e., each of the two curves $C_1, C_2$ is a so-called {\em triangle curve}). 

The nodes are the key for obtaining an obstructed moduli space, as noticed first, to our knowledge, by Burns and Wahl (\cite{b-w}). Thanks to their results  {\em one readily constructs examples (e.g., those of Segre) of surfaces of general
type with obstructed deformations} by constructing suitable nodal surfaces. Indeed, since $S_n$ has six nodes, by \cite{b-w}*{Corollary 1.3}   $h^1(\Theta_{S_n}) (:=\dim H^1(\Theta_{S_n})) \geq 6$ and it suffices to show that $S_n $ is in fact rigid.

Once we have a product-quotient surface coming from two triangle curves it is immediate that the equisingular deformations of the canonical model are trivial.
Then it has to be shown that none of the local deformations of the singularities lift to deformations of the canonical model. This can be deduced by the linear independence of certain elements of $H^2(\Theta_{S_n})$ (see Theorem \ref{howtodoit}, condition 2) that  have a simple explicit description in local coordinates due to Kas (cf. \cite{kas}).

In this paper we use only very special product-quotient surfaces: they are regular, their group $G$ is the Abelian group $(\ZZ / n \ZZ)^2$, where $n \geq 8$, even and not divisible by 3, and their singular models have only nodal singularities, but are indeed singular. 

Product-quotient surfaces have been extensively studied, especially for low invariants (like in the limit case $p_g =q=0$) and partial classification results are obtained in a long series of papers. We refer to \cite{fano}, \cite{isogenous}, \cite{BaCaGrPi12}, \cite{BaPi12}, \cite{BaPi16} and the literature there quoted.

For constructing the  higher dimensional rigid, but not infinitesimally rigid examples, we take the product of  $S_n$'s with a rigid manifold. 

More precisely, 
the second main result is:
\begin{theorem*}
Let $n \geq 8$ be an even integer such that $3 \nmid n$, and let $X$ be a compact complex rigid manifold.

Then $S_n \times X$ is rigid, but not infinitesimally rigid. 

In particular there are rigid, but not infinitesimally rigid, manifolds of dimension $d$ and Kodaira dimension $\kappa$ for all possible pairs $(d,\kappa)$ with $d \geq 5$ and $\kappa \neq 0,1,3$ and for $(d,\kappa)=(3,-\infty),(4,-\infty),(4,4)$.
\end{theorem*}

The paper is organized as follows.

In the first section we  collect some background material on deformation theory which will be used in the rest of the paper. We recall the notions of rigidity which are  relevant for our purposes and give a criterion for the minimal resolution of the singularities of a nodal surface to be rigid (cf. Theorem \ref{howtodoit}).

The second section is dedicated to Abelian covers and in particular to the proof of  formulae for the character decomposition of direct images of canonical and bicanonical sheaves, which will be stated in higher generality than needed in the sequel, because in the authors' opinion they are interesting themselves and useful for many applications. 
In the next section we give the construction of the infinite series of product-quotient surfaces $S_n$ and calculate their invariants. 

The fourth section is dedicated to the proof of our first main theorem, whereas the last section is dedicated to the higher dimensional examples, i.e., the proof of our second main result.

\section{A criterion to prove rigidity}\label{criterion}

In this section we shall recall the definitions of different concepts of rigidity of compact complex varieties, which were introduced and discussed in \cite{rigidity} and which are relevant for our paper. Then we briefly review results by Burns-Wahl (\cite {b-w}), Kas (\cite{kas}),  Pinkham (\cite{Pinkham}) and Catanese (\cite{Ca89}) which allow us to prove a criterion for rigidity (cf. Theorem \ref{howtodoit}).

Recall that two compact complex manifolds $X$ and $X'$ are said to be {\em  deformation equivalent} if and only if there is a 
proper smooth holomorphic map  $$f \colon \mathfrak{X}  \rightarrow \sB 
$$ 
where $\sB$ is a connected (possibly not reduced) complex space and there are points $b_0, b_0' \in \sB$ such that the fibres $X_{b_0} : = f^{-1} (b_0),
X_{b'_0} : = f^{-1} (b'_0)$ are respectively isomorphic to $X, X'$ ($X_{b_0} \cong X, X_{b'_0} \cong X'$). 

For the convenience of the reader we recall part of the notions of rigidity  given in \cite{rigidity}*{Definition 2.1}:
\begin{definition}\label{rigid} \
\begin{enumerate}
\item A compact complex manifold $X$ is said to be {\em globally rigid} if for any compact complex manifold $X'$, which is deformation equivalent to $X$, we have  an isomorphism $X \cong X'$.
\item  A compact complex manifold $X$ is said to be  {\em infinitesimally rigid} if 
\[
h^1(X, \Theta_X) = 0,
\]
where $\Theta_X$ is the sheaf of holomorphic vector fields on $X$.
\item  A compact complex manifold $X$ is said to be  {\em (locally) rigid} (or just {\em rigid}) if for each deformation of $X$,
\[
f \colon (\mathfrak{X},X)  \rightarrow (\sB, b_0)
\] 
there is an open neighbourhood $U \subset \sB$ of $b_0$ such that $X_t := f^{-1}(t) \cong X$ for all $t \in U$.

\end{enumerate}
\end{definition}

\begin{remark}
Observe that a globally/infinitesimally rigid compact complex manifold is (locally) rigid. If $X=S$ is a surface of general type, then $S$ is rigid if and only if $S$ is globally rigid due to the existence of the Gieseker moduli space for canonical models of surfaces of general type.
\end{remark}

Let $X$ be a {\em nodal} surface, i.e., a compact complex variety of dimension $2$ with $r$ singular points, all of type $A_1$.

Let $S \rightarrow X$ be the minimal resolution of  singularities of $X$ and let $E=E_1+\cdots+E_r$ be its exceptional locus. By \cite{b-w}*{Corollary 1.3} the local cohomology group $H^1_{E}(\Theta_S)=\oplus_1^r H^1_{E_i}(\Theta_S)$ embeds in $H^1(\Theta_S)$.

By \cite{b-w}*{Proposition 1.10}, for each $1 \leq i \leq r$,  $H^1_{E_i}(\Theta_S)$ has dimension $1$.  Let now $\nu_i \in X$ be the node, image of $E_i$, and let $\theta_i$ be a generator of $H^1_{E_i}(\Theta_S)$ seen as element in $H^1(\Theta_S)$. The {\em primary obstruction} $[\theta_i,\theta_i]$ belongs to   $H^2(\Theta_S)$, which is by Serre duality isomorphic to the dual of $H^0(\Omega^1_S \otimes \Omega^2_S)$.

Therefore we can see $[\theta_i,\theta_i]$ as a map $\alpha_{\nu_i} \colon H^0(\Omega^1_S \otimes \Omega^2_S) \rightarrow \CC$ which is explicitly described in \cite{kas}  as follows:
 a small neighbourhood $U_i$ of $\nu_i$ in $X$ is isomorphic to the quotient of a small disc $\Delta \subset \CC^2$, with coordinates $(z_1,z_2)$, by the involution $(z_1,z_2) \mapsto (-z_1,-z_2)$. This gives an inclusion of $H^0(\Omega^1_{U_i} \otimes \Omega^2_{U_i})$ into the invariant subspace 
$$ H^0(\Omega^1_{\Delta} \otimes \Omega^2_{\Delta})^+ \subset H^0(\Omega^1_{\Delta} \otimes \Omega^2_{\Delta}), $$ 
and thus every $\eta \in H^0 \left( \Omega^1_{S} \otimes \Omega_S^2 \right)$ can be locally written as 
\[
\eta = (f_1dz_1 + f_2dz_2) \otimes (dz_1 \wedge dz_2).
\]

Then, up to rescaling $\theta_i$

\begin{equation}\label{kas}
\alpha_{\nu_i}(\eta) = \left( \frac{\partial f_2}{\partial z_1} -\frac{\partial f_1}{\partial z_2} \right)(0,0).
\end{equation}

This allows to prove the following:

\begin{theorem}\label{howtodoit}
 Let $S \rightarrow X$ be the minimal resolution of the singularities of a nodal surface $X$. Assume that
 \begin{enumerate}
 \item $h^1(\Theta_X)=0$;
 \item the maps $\alpha_{\nu_i}$ associated to the nodes $\nu_i$ of $X$  locally described in (\ref{kas}) are linearly independent in $H^0(\Omega^1_S \otimes \Omega^2_S)^\vee$. 
 \end{enumerate}
 Then $S$ is rigid and $h^1(\Theta_S)$ equals the number of nodes of $X$.
\end{theorem}

We recall that for a normal variety the sheaf $\Theta_X$ is defined to be $Hom(\Omega^1_X,\hol_X)$, where $\Omega^1_X$ denotes the sheaf of K\"ahler differentials.
\begin{proof}
By \cite{Pinkham}*{Proof of Corollary [3]} we have that $H^1(\Theta_S)=H^1(\Theta_X)\oplus H^1_{E}(\Theta_S)$. Therefore, by condition 1, we can identify $H^1(\Theta_S)$ with $H^1_{E}(\Theta_S)\cong \CC^r$.

Then for every $\theta \in H^1(\Theta_S)$ there are $t_i \in \CC$ such that $\theta=\sum_1^r t_i\theta_i$. 

In \cite{kas}*{page 59} the cohomology class $\theta_i$ is represented explicitly  by a Dolbeault cocycle, a $\bar{\partial}-$closed  $(0, 1)-$form with values in the holomorphic tangent bundle of $S$.
The support of the representative given by Kas in loc. cit. is a compact subset of an arbitrarily small neighbourhood of $E_i$.

Then the Shouten bracket $[\theta_i,\theta_j]$ which is the composition of the exterior product of forms followed by the Lie bracket of vector fields, vanishes for all $ i \neq j$.
Because
\[
[\theta,\theta] = \left[ \sum_1^r t_i\theta_i,\sum_1^r t_i\theta_i \right]=\sum_1^r t_i^2 [\theta_i,\theta_i],
\] 
 the primary obstruction $[\theta,\theta]$ (considered as an element of  $H^0(\Omega^1_S \otimes \Omega^2_S)^\vee$) equals to  $\sum_1^r t_i^2 \alpha_{\nu_i}$.

Now consider the Kuranishi map $H^1(\Theta_S) \rightarrow H^2(\Theta_S)$ whose zero locus is the base of the Kuranishi family $\Def(S)$. Completing $\{\alpha_{\nu_i}\}$ to a basis of $H^2(\Theta_S)$, we obtain, considering only the first $r$ components of the Kuranishi map  
\[
\Def(S) \subset \left\{t_1^2+g_1(t_1,\ldots,t_r)=t_2^2+g_2(t_1,\ldots,t_r)=\ldots=t_r^2+g_r(t_1,\ldots,t_r)=0\right\}
\]
where the functions $g_j$ vanish at the origin of order at least three. In particular $\Def(S)$ is supported on the origin; in other words $S$ is rigid.
\end{proof}

\begin{remark}
We note that theorem (\ref{howtodoit}) is a generalization of the classical statement that infinitesimal rigidity implies rigidity. Indeed if $X$ is smooth in Theorem \ref{howtodoit} then  condition (2) is empty and the statement reduces ($S=X$) exactly to \cite{kodairamorrow}*{Theorem 3.2}.
\end{remark}

Let $S$ be a minimal surface of general type and let $X$ be its canonical model. Then $\Def(S)$ (respectively $\Def(X)$) denotes the base of the Kuranishi family of deformations of $S$ (respectively of $X$).

Let $G$ be a finite group acting faithfully on a smooth algebraic surface $Z$ and let $p \colon Z \rightarrow Z/G$ be the quotient map. If $p$ is unramified in codimension 1, then by \cite[Lemma 4.1]{Ca89} the natural map $(p_*\Theta_Z)^G \rightarrow \Theta_X$ is an isomorphism.

We  get thus the following special case of the more general \cite{Ca89}*{Corollary 1.20}
\begin{corollary}\label{fabweak}
Let $Z$ be a smooth algebraic surface, let $G$ be a finite group acting on $Z$ in such a way that the quotient map $p \colon Z \ra X = Z/G$ is unramified in codimension 1, and the singular locus of $X$ is a set of $r$ nodes.
If $h^1(\Theta_Z)^G=0$ and  condition 2 in Theorem \ref{howtodoit} holds for $X$, then $S$ is rigid and $h^1(\Theta_S)=r$.

In particular,  $\Def(S)$ is a  scheme of embedding dimension $r$ supported in a point.
\end{corollary}


\section{Character decomposition of the direct image of the bicanonical sheaf of an abelian cover}\label{abeliancovers}

Let $G$ be a finite Abelian group, acting on a normal complex variety $X$, such that $X/G$ is smooth, and denote by
 $\pi \colon X \rightarrow X/G =:Y$ the quotient map. 
 
 Following \cite{Par91}, the natural action of $G$ on $\pi_*\hol_X$ induces a splitting
 \[
 \pi_*\oo_X = \oplus_{\chi \in G^\vee} \cl^{-1}_\chi ,
 \]
 where $ \cl^{-1}_\chi =  (\pi_* \hol_X)^{(\chi)}:=\{f | g^*f=\chi(g)f \}$ are line bundles. 
 
 Differently from \cite{Par91}, we use  the additive notation for both the group $G$ and its group of characters $G^\vee$, since we find this notation more convenient for our main examples. 
 So the trivial character, mapping each $g \in G$ to $1$, is the character $0$. Obviously, $\cl_0\cong \oo_Y$ and, if $X$ is compact and connected, $0$ is the unique character such that $h^0(\cl_\chi^{-1})\neq 0$.
 
The branch locus of $\pi$ forms a divisor $D$ in $Y$ which we take with the reduced structure. Similarly the ramification locus of $\pi$ is the reduced  divisor $R$ supported on $\pi^{-1}(D)$. 
The inertia group $H_T$ of an irreducible component $T$ of $R$ is the subgroup  of $G$ of the elements fixing each point of $T$. Then (\cite{Par91}*{Lemmata 1.1 and 1.2}) $H_T$ is cyclic and there are a generator $\psi_T$ of $H_T^\vee$ and an uniformizing parameter $t$ for $\oo_{X,T}$ such that $h^*t=\psi_T(h)t$ for all $h\in H_T$. If two components $T_1$ and $T_2$ of $R$ map to the same component of $D$ then there is an element $g$ of $G$ such that $g(T_1)=T_2$; it follows, since $G$ is Abelian, that $(H_{T_1},\psi_{T_1})=(H_{T_2},\psi_{T_2})$.
This induces a decomposition of the branch and of the ramification divisors respectively
\begin{align*}
R=&\sum_{\substack{H \leqslant G \text{ cyclic }\\ \psi \text{ generating } H^\vee}} R_{H,\psi}&
D=&\sum_{\substack{H \leqslant G \text{ cyclic }\\ \psi \text{ generating } H^\vee}} D_{H,\psi} ,
\end{align*}
 where $R_{H,\psi}$ is the union of the components $T$ of $R$ with inertia group $H$ and such that $\psi_T$ equals $\psi$, and $D_{H,\psi}$ is its image.
 
\begin{definition}
The {\em building data} of the $G$-cover $\pi$ are
\begin{itemize}
\item the line bundles $\cl_\chi$, $\chi \in G^\vee$,
\item the reduced divisors $D_{H,\psi}$ on $Y$, where $H$ varies over all cyclic subgroups of $G$ and $\psi$ over the generators of $H^\vee$.
\end{itemize}
\end{definition}
 
 Fix a pair $(H,\psi)$ as above. For every character $\chi$ in $G^\vee$ there is a unique integer $0\leq i_{\chi,\psi} \leq |H|-1$ such that $\chi_{|H}=\psi^{i_{\chi,\psi}}$. For each pair of characters $\chi,\chi' \in G^\vee$ set
 \[
 \epsilon^{H,\psi}_{\chi,\chi'}=
 \begin{cases}
0&\text{ if }  i_{\chi,\psi}+i_{\chi',\psi} \leq |H|-1,\\
1&\text{ otherwise}.
 \end{cases}
 \]

 The following is \cite{Par91}*{Theorem 2.1}.
 \begin{theorem}[Pardini] 
 Let $X$ be a normal variety and $Y$ a smooth variety.
 Let $\cl_\chi$, $D_{H,\psi}$ be the building data of a $G-$cover $\pi \colon X \rightarrow Y$. Then
\begin{equation}\label{RitaEquations}
\cl_{\chi} \otimes \cl_{\chi'} \cong \cl_{\chi+\chi'} \otimes   \hol_Y \left( \sum_{\substack{H \leqslant G \text{ cyclic }\\ \psi \text{ generating } H^\vee}}  \epsilon^{H,\psi}_{\chi,\chi'} D_{H,\psi} \right) .
\end{equation}
Conversely, to any set of data $\cl_\chi$, $D_{H,\psi}$, satisfying (\ref{RitaEquations}) we can associate a $G-$cover $\pi \colon X \rightarrow Y$ with $X$ normal whose building data are exactly $\cl_\chi$ and $D_{H,\psi}$. Moreover, if $Y$ is complete, then the building data determine the cover up to isomorphisms of Galois covers. 
 \end{theorem}

  Note that, if $\chi$ has order $n$, then $\cl_{n\chi}=\cl_{0}\cong \hol_Y$. Applying recursively (\ref{RitaEquations}) we obtain
 \begin{equation}\label{ImportantEquations}
 \cl_{\chi}^{\otimes n} \cong \hol_Y \left( \sum_{\substack{H \leqslant G \text{ cyclic }\\ \psi \text{ generating } H^\vee}}  i_{\chi,\psi} D_{H,\psi} \right)
 \end{equation}
 
In particular, if $\Pic(Y)$ has trivial $n-$torsion and we know all the divisors $D_{H,\psi}$, we can deduce $\cl_\chi$ by (\ref{ImportantEquations}).

 Let us consider now the special case when $G=\left( \ZZ/n\ZZ \right)^k$, $Y=\PP^1$.
Then $n\chi=0$ for all $\chi \in G^\vee$. 
 Since $\Pic(\PP^1)\cong \ZZ$  a line bundle on $\PP^1$ is determined up to isomorphism by its degree, which we may compute by (\ref{ImportantEquations}) obtaining
 \begin{corollary}\label{TheCaseZ/nZk}
  If $\pi \colon X \rightarrow \PP^1$ is a $\left( \ZZ/n\ZZ \right)^k-$cover  then
\[
 \cl_{\chi}\cong \hol_{\PP^1}\left( \frac1n \sum_{\substack{H \leqslant G \text{ cyclic }\\ \psi \text{ generating } H^\vee}}  i_{\chi,\psi} \deg D_{H,\psi} \right).
\]
\end{corollary}

Since $X$ is normal, it has a canonical Weil divisor $K_X$. If $X^\circ$ is the smooth locus of $X$, $\hol_X(K_X)=i_*(\omega_{X^\circ})$ where $i \colon X^\circ \hookrightarrow X$ is the inclusion and $\omega_{X^\circ}$ is the dualizing sheaf of $X^\circ$, the sheaf of the holomorphic $\dim X$-forms on $X^\circ$. Recall that if $X$ is Gorenstein in codimension $1$, $\hol_X(K_X)$ is the dualizing sheaf $\omega_X$: we will indeed apply the forthcoming Proposition \ref{canonicalsplitting} to nodal surfaces, which in fact are Gorenstein.  

$G$ acts on  $\pi_* \omega_X$ via the unique extension of the action given by the  pull-back of holomorphic differential forms on the smooth locus, inducing a direct sum decomposition in eigensheaves according to the characters as follows
 \[
 \pi_*\hol_X(K_X) = \oplus_{\chi \in G^\vee} (\pi_*\hol_X(K_X))^{(\chi)} ,
 \]

 We give the following generalization of a result of Pardini (\cite{Par91}):
 \begin{proposition}\label{canonicalsplitting}  Let  $\pi \colon X \rightarrow Y$ be a $G-$cover, with $X$ a normal variety and $Y$ a smooth variety, whose building data are
 $\cl_\chi$, $D_{H,\psi}$. Then
 \[
 (\pi_*\hol_X(K_X))^{(\chi)}\cong \omega_{Y} \otimes \cl_{-\chi}.
 \]
 \end{proposition}
 \begin{proof}
If  $X$ is smooth this is  \cite{Par91}*{Proposition 4.1, c)}.
 
In the general case, we consider the inclusion $i \colon X^{\circ} \hookrightarrow X$, we set $Y^{\circ}=\pi\left(X^{\circ}\right)$. Then, since $X^\circ$ is smooth, the statement holds true for the induced $G-$cover $X^\circ \rightarrow Y^\circ$, i.e., 
\[
(\pi_*\omega_{X^\circ})^{(\chi)}\cong \omega_{Y^\circ}\otimes (\cl_{-\chi})_{|Y^\circ}.
\]  
Consider the inclusion $i^Y \colon Y^{\circ} \hookrightarrow Y$. Then
\[
(\pi_*\hol_X(K_X))^{(\chi)}=(\pi_*i_*\omega_{X^\circ})^{(\chi)}=(i^Y_*\pi_*\omega_{X^\circ})^{(\chi)}=i^Y_*(\pi_*\omega_{X^\circ})^{(\chi)}=i^Y_*(\omega_{Y^\circ}\otimes (\cl_{-\chi})_{|Y^\circ}).
\]

Finally we conclude that  $i^Y_*(\omega_{Y^\circ}\otimes (\cl_{-\chi})_{|Y^\circ}) \cong \omega_{Y}\otimes \cl_{-\chi}$, because they are isomorphic on $Y^\circ$, $Y$ is smooth and the complement of $Y^\circ$ has codimension at least $2$.
 \end{proof}
 
The $G$-action on $\pi_*\hol_X(K_X)$ induces $G$-actions on $\pi_*\hol_X(kK_X)$ for all $k \in \ZZ$ and then splittings
 \[
 \pi_*\hol_X(kK_X) = \oplus_{\chi \in G^\vee} (\pi_*\hol_X(kK_X))^{(\chi)} .
 \]
These are all line bundles and we  compute them explicitly for the case $k=2$.

\begin{proposition}\label{bicanonicalsplitting}
 Let  $\pi \colon X \rightarrow Y$ be a $G-$cover, with $X$ a normal variety and $Y$ a smooth variety, whose building data are
 $\cl_\chi$, $D_{H,\psi}$. Then
 \[
 \left( \pi_*\hol_X(2K_X) \right)^{(\chi)}\cong 
  (\pi_*\hol_X(K_X))^{(\chi)} \otimes \omega_Y \left( \sum_{\chi_{|H} \neq \psi} D_{H,\psi} \right)
\cong  \omega_{Y}^{\otimes 2} \otimes \cl_{-\chi} \left( \sum_{\chi_{|H} \neq \psi} D_{H,\psi} \right).
 \]
 \end{proposition}

 \begin{proof} By the argument of the proof of Proposition \ref{canonicalsplitting} it is enough if we prove the statement for $X$ smooth.

Therefore  we assume  $X$ to be smooth and  notice that all sheaves $ ( \pi_*\hol_X(2K_X))^{(\chi)}$ are line bundles.  Hence we may prove the statement by showing that the cokernel of the injective morphism 
\begin{equation}\label{1+1=2}
( \pi_*\hol_X(K_X))^{(\chi)} \otimes ( \pi_*\hol_X(K_X))^G \rightarrow ( \pi_*\hol_X(2K_X))^{(\chi)},
 \end{equation}
obviously supported on $D$, has multiplicity $1$ in each $D_{H,\psi}$ for $\chi_{|H} \neq \psi$ and $0$ for $\chi_{|H} = \psi$.

Take a general point $p$ of $D$. Then $p$ belongs to exactly one of the $D_{H,\psi}$ and the cover $\pi$ may  locally be written as 
  \[(x_1,\ldots,x_{d-1},t) \mapsto (y_1,\ldots,y_{d-1},y_d)=(x_1,\ldots,x_{d-1},t^{|H|})\]
  where $x_1, \ldots,x_{d-1},t$ are local coordinates near a preimage $q$ of $p$ and $h\in H$ acts on them by fixing all $x_j$ and multiplying $t$ by $\psi(h)$.
  
  Then 
a local generator  of $(\pi_*\hol_X(K_X))^G$ is $t^{|H|-1}dx_1 \wedge \cdots \wedge dx_{d-1} \wedge dt$.  

Similarly,  local generators of $ (  \pi_*\hol_X(kK_X))^{(\chi)} $ are 
\[
t^{a_k}(dx_1 \wedge \ldots \wedge dx_{d-1}\wedge dt)^{\otimes k},  \text{ for some } 0\leq a_k \leq |H|-1.
\]

Note that $a_k=0 \Leftrightarrow \chi_{|H}=\psi^k$. In particular, if $a_1=0$, equivalently if $\chi_{|H}=\psi$, the tensor product of the given local generators  of $( \pi_*\hol_X(K_X))^G$ and $ ( \pi_*\hol_X(K_X))^{(\chi)} $ maps to the given local generator of $ ( \pi_*\hol_X(2K_X))^{(\chi)} $, and then the map (\ref{1+1=2}) is an isomorphism in a neighbourhood of $p$.

On the other hand, if $a_1 \neq 0$, equiv. if $\chi_{|H} \neq \psi$, the same tensor product maps to $t^{|H|}$ times the given local generator of $ ( \pi_*\hol_X(2K_X))^{(\chi)}$. Now, $t^{|H|}$ is the pull-back of a local generator of the ideal of $D_{H,\psi}$ at $p$, and this implies the result.
 \end{proof}


\section{An infinite series of product-quotient surfaces}\label{sec:curve}

 The aim of this section is to construct for each even $n \in \NN$, such that $3 \nmid  n$, a surface $X_n$ of general type having  6 nodes as singularities. 
We consider the group $G:=\left( \ZZ/n\ZZ \right)^2$.

For $n \geq 2$, let  $C^{(n)}$ be the {\em Fermat curve of degree $n$}, i.e., 
$$C^{(n)}:=\left\{\sum_{j=0}^2 x_j^n=0\right\} \subset {\mathbb P}_{\mathbb C}^2,$$ 
a smooth plane curve of genus $g(C^{(n)})=
1+\frac{n(n-3)}2$.

On  $C^{(n)}$  we consider the $G-$action 
$$(a_1,a_2)(x_0:x_1:x_2)=(x_0:e^{a_1\frac{2\pi i}n}x_1:e^{a_2\frac{2\pi i}n}x_2).$$
This action has exactly three orbits of cardinality different from $n^2$, the hyperplane sections $C^{(n)}\cap \{x_j=0\}$ with respective stabilizers $\langle (1,1)\rangle$, $\langle (1,0)\rangle$, $\langle (0,1)\rangle$, all isomorphic to $\ZZ / n \ZZ$.
Then the degree of the branch divisor $D$ of $C^{(n)} \ra C^{(n)}/G$ equals three and, by Hurwitz formula $C^{(n)}/G \cong {\mathbb P}^1$.  Thus we set the three branch points to be respectively $1$, $0$ and $\infty$.

We may compute the decomposition of $D$ as sum of $D_{H,\psi}$ using, since each point of the ramification divisor lies in one of the hyperplane sections, say ${x_j=0}$, an uniformizing parameter $\frac{x_j}{x_k}$, $k \neq j$. We obtain that, setting $g_0=(1,0)$, $g_\infty=(0,1)$ and $g_1=(-1,-1)$,  each branch point $p\in\{0,1,\infty\}$ is the branch divisor $D_{H_p,\psi_p}$, where $H_p=\langle g_p \rangle$ and $\psi_p \colon H_p \rightarrow \CC^*$ is the character mapping $g_p$ to $\eta:=e^{\frac{2\pi i}{n}}$.

\begin{remark} \

1) We recall that giving a $\left( \ZZ/n\ZZ \right)^2$-Galois cover $p \colon C^{(n)} \rightarrow \PP^1$ branched on $\{0,1,\infty\}$ as above is essentially equivalent to give generators $g_0, g_1, g_{\infty}$ of  $\left( \ZZ/n\ZZ \right)^2$ such that  $g_0+g_1+g_\infty=0$. For details (in a much more general setting) we refer to \cite{BaCaGrPi12}*{page 1002}.

2) A finite Galois cover of $\PP^1$ branched on $\{0,1,\infty\}$ is called a {\em triangle curve}.

\end{remark}

\begin{notation}
For describing the  characters in  $G^\vee$ we fix a bijection $\ZZ/ n\ZZ \ra \{0,1, \ldots , n-1\}$, in other words if we write a character $\chi$ as $(\alpha, \beta)$ we automatically assume that $0 \leq \alpha,\beta \leq n-1$.
\end{notation}

Then
\[
\forall (a,b) \in G \ \ \chi(a,b)= (\alpha,\beta)(a,b)=\eta^{\alpha a + \beta b},
\]
whence
\begin{align*}
\chi_{|H_0}&=\psi_0^\alpha , &
\chi_{|H_\infty}&=\psi_\infty^\beta, &
\chi_{|H_1}&=\psi_1^{-\alpha-\beta}.
 \end{align*}

Splitting $p_* \oo_{C^{(n)}}=\oplus_{\chi \in G^\vee} \cl_\chi^{-1}$  as sum of line bundles according to the action of $G$, 
by Corollary \ref{TheCaseZ/nZk}
\[
\cl_{(\alpha,\beta)}\cong \oo_{\PP^1}  \left(\left\lceil \frac{\alpha + \beta  }n \right\rceil \right) .
\]
where  $\lceil \cdot \rceil$ denotes the usual  ceiling function with integral values. In particular:
\begin{align*}
\cl_{(0,0)}&=\oo_{\PP^1},\\
 \cl_{(\alpha,\beta)}&=\oo_{\PP^1}(1) \text { if } 1 \leq \alpha + \beta \leq n,\\
\cl_{(\alpha,\beta)}&=\oo_{\PP^1}(2) \text { if }  \alpha + \beta \geq n+1.
 \end{align*}
 Therefore, by Proposition \ref{canonicalsplitting}, and observing that if $\chi = (\alpha, \beta)$, then (if $\alpha, \beta \neq 0$) $-\chi = (n-\alpha, n-\beta)$, we obtain that the summands of $p_*\omega_{C^{(n)}}$ are 
 \begin{align}
\notag (p_*\omega_{C^{(n)}})^{(0,0)}&=\oo_{\PP^1}(-2),\\
\label{canonicalsplittingcurve} (p_*\omega_{C^{(n)}})^{(\alpha,\beta)}&=\oo_{\PP^1} \text { if } \alpha,\beta \neq 0, \alpha + \beta \leq n-1,\\
\notag (p_*\omega_{C^{(n)}})^{(\alpha,\beta)}&=\oo_{\PP^1}(-1) \text { else.}
 \end{align}
 
\begin{figure}
\begin{tikzpicture}[scale=.65]
  \foreach \y in {0,1,2,3} \node at(-1,\y)  {$\y$};
  \foreach \y in {1,2} \node at (-1, {11-\y}) {$n$-$\y$};
  \foreach \x in {0,1,2,3} \node at (\x,-1)  {$\x$};
  \foreach \x in {1,2} \node at (11-\x ,-1) { $n$-$\x$};

	\node[red] at (0,0)  {-2};
    \foreach \x in {1,2,3,4,6,7,8,9,10} \node[blue] at (\x,0)  {-1};
 \foreach \y in {1,2,3,4,6,7,8,9,10} \node[blue] at (0,\y)  {-1};

    \foreach \x in {1,2,3,4,6,7,8,9} \node at (\x,1){0};
  \foreach \x in {10} \node[blue] at (\x,1) {-1};

    \foreach \x in {1,2,3,4,6,7,8} \node at (\x,2){0};
\foreach \x in {9,10} \node[blue] at(\x,2) {-1};

    \foreach \x in {1,2,3,4,6,7} \node at(\x,3) {0};
\foreach \x in {8,9,10} \node[blue] at (\x,3) {-1};

    \foreach \x in {1,2,3,4,6} \node at (\x,4) {0};
\foreach \x in {7,8,9,10} \node[blue] at (\x,4) {-1};

	\node[blue] at (0,5)  {*};
    \foreach \x in {1,2,3,4, 5} \node at(\x,5) {*};
      \foreach \x in {6,7,8,9,10} \node[blue] at(\x,5) {*};
    
    \node[blue] at (5,0)  {*};
        \foreach \y in {1,2,3,4} \node at(5,\y) {*};
         \foreach \y in {6,7,8,9,10} \node[blue] at(5,\y) {*};

    \foreach \x in {1,2,3,4} \node at (\x,6) {0};
\foreach \x in {6,7,8,9,10} \node[blue] at (\x,6) {-1};

    \foreach \x in {1,2,3} \node at (\x,7) {0};
\foreach \x in {4,6,7,8,9,10} \node[blue] at (\x,7) {-1};

    \foreach \x in {1,2} \node at (\x,8) {0};
\foreach \x in {3,4,6,7,8,9,10} \node[blue] at(\x,8) {-1};

    \foreach \x in {1} \node at (\x,9) {0};
    \foreach \x in {2,3,4,6,7,8,9,10} \node[blue] at (\x,9) {-1};

    \foreach \x in {1,2,3,4,6,7,8,9,10} \node[blue] at (\x,10) {-1};
    \end{tikzpicture}
   \caption{The degrees of $(p_*\omega_{C^{(n)}})^{(\alpha,\beta)}$}\label{fig:omega1}
\end{figure}
 
 \begin{remark}
This implies in particular
 \[
 H^0(\omega_{C^{(n)}})
 =\bigoplus_{\substack{\chi=(\alpha,\beta)\\ \alpha +\beta \leq n-1\\ \alpha,\beta \geq 1}} \omega^{(\chi)} \CC
  \]
 where $\omega^{(\chi)}$ is a global form such that $\forall g \in G$, $g^*\omega^{(\chi)}= \chi(g) \omega^{(\chi)}$.

Observe that the divisor of $\omega^{(\chi)}$ is 
\[
\left( \omega^{(\chi)} \right)=(\alpha-1)R_{H_0,\psi_0}+(\beta-1)R_{H_\infty,\psi_\infty}+(n-\alpha-\beta-1)R_{H_1,\psi_1}.
\]

\end{remark}

Then  $(p_*\omega^{\otimes 2}_{C^{(n)}})^{(\chi)}$ can be determined by Proposition \ref{bicanonicalsplitting}.
\begin{proposition}\label{charbicanonical}
If $n \geq 4$, then
\begin{align*}
(p_*\omega^{\otimes 2}_{C^{(n)}})^{(\alpha,\beta)}=&\oo_{\PP^1}(-1), &&\text{ if } (\alpha,\beta)\in \{0,1\}^2 \cup \{(0,n-1),(n-1,0)\} \\
&&& \cup \{(1,n-1),(n-1,1)\} \cup \{(1,n-2),(n-2,1)\},\\
(p_*\omega^{\otimes 2}_{C^{(n)}})^{(\alpha,\beta)}=&\oo_{\PP^1}(1), &&\text { if } \alpha,\beta \geq 2, \alpha + \beta \leq n-2,\\
(p_*\omega^{\otimes 2}_{C^{(n)}})^{(\alpha,\beta)}=&\oo_{\PP^1}, &&\text{ else.}
\end{align*}
\end{proposition}

\begin{figure}
\begin{tikzpicture}[scale=.65]

  \foreach \y in {0,1,2,3} \node at(-1,\y)  {$\y$};
  \foreach \y in {1,2,3} \node at (-1, {11-\y}) {$n$-$\y$};
  \foreach \x in {0,1,2,3} \node at (\x,-1)  {$\x$};
  \foreach \x in {1,2,3} \node at (11-\x ,-1) {$n$-$\x$};

 \foreach \x in {0,1,5,6,7,8,9,10} \node at(\x,5) {*};
 \foreach \x in {2,3,4} \node[red] at(\x,5) {*};

 \foreach \y in {0,1,6,7,8,9,10} \node at(5,\y) {*};
 \foreach \y in {2,3,4} \node[red] at(5,\y) {*};

\node[blue] at (0,0)  {-1}; \node[blue] at (0,1)  {-1};
\node[blue] at (1,0)  {-1}; \node[blue] at (1,1)  {-1};

    \foreach \x in {2,3,4,6,7,8,9} \node[black] at (\x,0)  {0};
\node[blue] at (10,0)  {-1};    
     \foreach \y in {2,3,4,6,7,8,9} \node[black] at (0,\y)  {0};
  \node[blue] at (0,10)  {-1};    
    
       \foreach \x in {2,3,4,6,7,8} \node[black] at (\x,1){0};
  \foreach \x in {9,10} \node[blue] at (\x,1) {-1};

       \foreach \y in {2,3,4,6,7,8} \node[black] at (1,\y){0};
  \foreach \y in {9,10} \node[blue] at (1,\y) {-1};

    \foreach \x in {2,3,4,6,7} \node[red] at (\x,2){1};
\foreach \x in {8,9,10} \node[black] at(\x,2) {0};

    \foreach \x in {2,3,4,6} \node[red] at (\x,3){1};
\foreach \x in {7,8,9,10} \node[black] at(\x,3) {0};

 \foreach \x in {2,3,4} \node[red] at (\x,4){1};
\foreach \x in {6,7,8,9,10} \node[black] at(\x,4) {0};

 \foreach \x in {2,3} \node[red] at (\x,6){1};
\foreach \x in {4,6,7,8,9,10} \node[black] at(\x,6) {0};

 \foreach \x in {2} \node[red] at (\x,7){1};
\foreach \x in {3,4,6,7,8,9,10} \node[black] at(\x,7) {0};

\foreach \x in {2,3,4,6,7,8,9,10} \node[black] at(\x,8) {0};
\foreach \x in {2,3,4,6,7,8,9,10} \node[black] at(\x,9) {0};
\foreach \x in {2,3,4,6,7,8,9,10} \node[black] at(\x,10) {0};
\end{tikzpicture}
   \caption{The degrees of  $(p_*\omega^2_{C^{(n)}})^{(\alpha,\beta)}$}\label{fig:omega2}
\end{figure}

\begin{proof}
By Proposition \ref{bicanonicalsplitting}, $(p_*\omega^{\otimes 2}_{C^{(n)}})^{(\alpha,\beta)}\cong (p_*\omega_{C^{(n)}})^{(\alpha,\beta)} \otimes \oo_{\PP^1}(\delta-2)$ where $\delta$ is the degree of the divisor $D$ with $0\leq D \leq p_0+p_1+p_\infty$ such that
\begin{align*}
p_0 \leq D &\Leftrightarrow \alpha \neq 1,&
p_\infty \leq D &\Leftrightarrow \beta \neq 1,&
p_1 \leq D &\Leftrightarrow \alpha+\beta \neq n-1.
\end{align*}

This leads us to consider the three lines $\alpha=1$, $\beta=1$ and $\alpha+\beta=n-1$  and the  triangle they form.

In the three vertices $(1,1),(1,n-2),(n-2,1)$ of the triangle $\delta=1$.  By (\ref{canonicalsplittingcurve}) they all have 
$(p_*\omega_{C^{(n)}})^{(\alpha,\beta)} \cong \oo_{\PP^1}$ and therefore $(p_*\omega^{\otimes 2}_{C^{(n)}})^{(\alpha,\beta)} \cong \oo_{\PP^1}(-1)$. 

In the remaining points of these three lines, $\delta=2$ and then  $(p_*\omega_{C^{(n)}})^{(\alpha,\beta)} \cong (p_*\omega^{\otimes 2}_{C^{(n)}})^{(\alpha,\beta)} $. 

By (\ref{canonicalsplittingcurve}), if $\chi \in \{(1,0), (0,1), (0,n-1), (n-1,0),(1,n-1), (n-1,1) \}$, then $(p_*\omega^{\otimes 2}_{C^{(n)}})^{(\alpha,\beta)} \cong  \oo_{\PP^1}(-1)$, else $(p_*\omega^{\otimes 2}_{C^{(n)}})^{(\alpha,\beta)} \cong  \oo_{\PP^1}$.

Finally, outside the three lines we have $\delta=3$.

If $\chi$ is inside the triangle, then by (\ref{canonicalsplittingcurve})  $(p_*\omega_{C^{(n)}})^{(\alpha,\beta)} \cong \oo_{\PP^1}$ and $(p_*\omega^{\otimes 2}_{C^{(n)}})^{(\alpha,\beta)} \cong  \oo_{\PP^1}(1)$.

For $(\alpha,\beta) =(0,0)$, then $(p_*\omega_{C^{(n)}})^{(0,0)} \cong \oo_{\PP^1}(-2)$ whence $(p_*\omega^{\otimes 2}_{C^{(n)}})^{(0,0)} \cong \oo_{\PP^1}(-1)$.

In the remaining cases $(p_*\omega_{C^{(n)}})^{(\alpha,\beta)} \cong \oo_{\PP^1}(-1)$ and $(p_*\omega^{\otimes 2}_{C^{(n)}})^{(\alpha,\beta)} \cong \oo_{\PP^1}$. 
\end{proof}

From now on we fix $n\geq 4$, even and $3\nmid n$ and we denote $C^{(n)}$ simply by $C$.

We define the following action of $G$ on $Z:=C \times C$: for $(a,b) \in  G$, for $(z_1,z_2) \in C \times C$
$$
(a,b)(z_1,z_2) := \left( (a,b)z_1, (a',b')z_2\right),
$$
where
$$
A \begin{pmatrix}
a'\\
b'
\end{pmatrix} := 
\begin{pmatrix}
a\\
b 
\end{pmatrix}  \ , \ A:= \begin{pmatrix}
1&-2\\
2&-1
\end{pmatrix}.
$$

Since $3 \nmid n$, $A \in GL\left( 2, \ZZ/n\ZZ 	\right)$.

\begin{remark}\label{monodromies}
In other words, we take two different $G$-actions on the same curve $C$, differing by an automorphism of $G$.

The two actions give isomorphic covers $p_j \colon C_j \cong C \rightarrow C_j/G \cong \PP^1$, $j=1,2$ branched on $\{0,1,\infty\}$ with different local monodromies: the local monodromies of $p_2$ are the images of the local monodromies of $p_1$ by the matrix $A$.

More precisely the cover $p_1$ has local monodromies 
$g_0=(1,0)$ at $0$, $g_\infty=(0,1)$ at $\infty$ and $g_1=(-1,-1)$ at $1$, whereas the cover $p_2$ has local monodromies  $h_0=(1,2)$ at $0$, $h_\infty=(-2,-1)$ at $\infty$ and $h_1=(1,-1)$ at $1$.

\end{remark}

\begin{remark}\label{twisting}

The line bundles $((p_1)_* \omega^{\otimes k}_{C_1})^{\chi}$, $k \in \{1,2\}$, are exactly the line bundles $(p_* \omega^{\otimes k}_{C^{(n)}})^{\chi}$ computed in (\ref{canonicalsplittingcurve})  and Proposition \ref{charbicanonical}, see respectively Figure \ref{fig:omega1} and \ref{fig:omega2}. 

Instead, for the action on the second factor, we observe that 
$$
((p_2)_* \omega^{\otimes k}_{C_2})^{\chi} \cong (p_* \omega^{\otimes k}_{C^{(n)}})^{\chi'}
$$
where $\chi'=\chi \circ A$. If $\chi=(\alpha,\beta)$, we set for later convenience $\chi'=(-\alpha',-\beta')$ so
\begin{equation}\label{twist}
\begin{pmatrix} \alpha' \\ \beta' \end{pmatrix} \equiv
- {}^tA 
\begin{pmatrix} \alpha \\ \beta \end{pmatrix} 
\mod n
.
\end{equation}
\end{remark}

Let be $X_n:= (C \times C)/G$ and let $\rho \colon S_n \rightarrow X_n$ be the minimal resolution of the singularities of $X_n$.

\begin{proposition}
For each even $n \geq 4$, not divisible by $3$, $X_n$ has six nodes as only singularities. $S_n$ is a minimal regular surface of general type with invariants:
\begin{align*}
K^2_{S_n}&=2(n-3)^2,\\
\chi(\oo_{S_n})&=\frac{n^2-6n+12}4,\\
p_g({S_n}) &=\left( \frac{n}2-2 \right)\left( \frac{n}2-1 \right).
\end{align*}
\end{proposition}

\begin{proof}
Note that $\langle g_p \rangle \cap \langle h_q \rangle=\{(0,0)\}$ for $p \neq q$ whereas $\langle g_p \rangle \cap \langle h_p \rangle \cong \ZZ/2\ZZ$. More precisely,  $\langle g_p \rangle \cap \langle h_p \rangle =\left\langle s_p \right\rangle$ where $s_0=\left( \frac{n}2 ,0 \right)$, $s_\infty = \left( 0,\frac{n}2 \right)$, $s_1 = 	\left(\frac{n}2 ,\frac{n}2 \right) $.

For all $p \in \{0,1,\infty \}$ there are $n^2$ points of $C \times C$ lying over  $(p,p)\in \PP^1 \times \PP^1$. Since $\langle g_p \rangle \cap \langle h_p \rangle$ has order $2$, they split in $\frac{n^2}{\frac{n^2}2}=2$ orbits, so producing each $2$ nodes on the quotients $X_n:=(C \times C)/G$. Hence  $X_n$ has exactly $3 \cdot 2=6$ nodes. 

By \cite{BaCaGrPi12} $S_n$ is regular and
\[
K^2_{S_n}=\frac{8(g_1-1)(g_2-1)}{|G|}=\frac{8 \left( \frac{n(n-3)}2\right)^2}{n^2}=2(n-3)^2,
\]
\[
\chi(\oo_{S_n})=\frac{K_{S_n}^2+6}8=\frac{2n^2-12n+24}8=\frac{n^2-6n+12}4,
\]
\[
p_g({S_n})=\frac{n^2-6n+8}4=\frac{(n-4)(n-2)}4=\left( \frac{n}2-2 \right)\left( \frac{n}2-1 \right).
\]

\end{proof}

\begin{remark}
$X_n$ are (singular models of) so-called {\em product-quotient} surfaces, introduced  in \cite{BaPi12}, \cite{BaCaGrPi12}.
\end{remark}


\section{The deformations of \texorpdfstring{$S_n$}{Sn} and \texorpdfstring{$X_n$}{Xn}}

This section is dedicated to the proof of our main result.

\begin{theorem}\label{main}
Let $n \in \NN$ be an even number $\geq 8$, not divisible by $3$. Then $S_n$ is rigid and $h^1(\Theta_{S_n}) =6$.

In particular, $S_n$ is an infinite series of minimal regular surfaces of general type with unbounded invariants which are rigid, but not infinitesimally rigid.
\end{theorem}

\begin{proof}
Set $Z:= C \times C$ and let  $\pi\colon Z \rightarrow X_n =Z/G$ be the quotient map.  

Since $C$ is a triangle curve, $h^1(\Theta_C)^G= 0$ and then $h^1(\Theta_Z)^G=2h^1(\Theta_C)^G=0$.

Since $\pi$ is unramified in codimension 1, the result follows by Corollary \ref{fabweak} if condition 2 in Theorem \ref{howtodoit} holds for $X_n$.

This will be proven in Proposition \ref{surjective}.
\end{proof}

The rest of the section is devoted to the proof that, under our assumptions, condition 2 in Theorem \ref{howtodoit} holds for $X_n$. 

More precisely,  we need to check that the six maps $\alpha_{\nu} \colon H^0(\Omega^1_{S_n} \otimes \Omega^2_{S_n})  \rightarrow \CC$ associated to the nodes of $X$  locally described in (\ref{kas}) are linearly independent in $H^0(\Omega^1_{S_n} \otimes \Omega^2_{S_n})^\vee$. 

We notice that it is sufficient to find a vector subspace $V \subset  H^0(\Omega^1_{S_n} \otimes \Omega^2_{S_n})$ such that the restrictions to $V$ of our six maps are linearly independent in $V^\vee$. Since this implies $\dim V \geq 6$, we will choose six linearly independent vectors in $H^0(\Omega^1_{S_n} \otimes \Omega^2_{S_n})$ and take the vector subspace they generate.

First observe that by \cite{Ca89}*{Proposition 1.6} we know that
$$
H^0(\Omega^1_{S_n} \otimes \Omega^2_{S_n}) \cong H^0(\Omega^1_{C \times C} \otimes \Omega^2_{C \times C}  )^G.
$$

By the K\"unneth formula (cf. \cite{kaup}) we have:
\begin{equation}\label{kuennethdecomposition}
H^0(\Omega^1_{C \times C} \otimes \Omega^2_{C \times C}  )
=   \left(
H^0(\omega^{\otimes 2}_{C})\otimes
H^0(\omega_{C})\right)
\oplus \left(
H^0(\omega_{C}) \otimes
H^0(\omega^{\otimes 2}_{C})
\right) .
\end{equation}

Under this isomorphism, a section 
$
f_1(z_1)(dz_1)^2 \otimes f_2(z_2) dz_2 + g_1(z_1)dz_1 \otimes g_2(z_2)(dz_2)^2 
$
in $ \left(
H^0(\omega^{\otimes 2}_{C})\otimes
H^0(\omega_{C})\right)
\oplus \left(
H^0(\omega_{C}) \otimes
H^0(\omega^{\otimes 2}_{C})
\right)$
corresponds to  the section
\[(f_1f_2dz_1 + g_1g_2dz_2) \otimes (dz_1 \wedge dz_2) \in H^0(\Omega^1_{C \times C} \otimes \Omega^2_{C \times C}  ).
\] 
This will be useful to compute the Kas maps (\ref{kas}) in the six vectors we are going to choose.

The group $G$ acts on both sides of equation (\ref{kuennethdecomposition}) producing a K\"unneth decomposition of each eigenspace $H^0(\Omega^1_{C \times C} \otimes \Omega^2_{C \times C}  )^{(\chi)}$. The result for the $G$-invariant part is, by Remark \ref{twisting},
\begin{multline}\label{kuennethdecomposition2}
H^0(\Omega^1_{C \times C} \otimes \Omega^2_{C \times C}  )^G=\\
=\bigoplus_{\chi \in G^\vee} \left(  \left(
H^0(\omega^{\otimes 2}_{C})^{(\chi)} \otimes
H^0(\omega_{C})^{(-\chi')} \right)
\oplus \left(
H^0(\omega_{C})^{(\chi)} \otimes
H^0(\omega^{\otimes 2}_{C})^{(-\chi')}
\right) \right)
\end{multline}
where, if $\chi=(\alpha,\beta)$, $-\chi'=(\alpha',\beta')$ with $\alpha'$, $\beta'$ as in (\ref{twist}).

If we find six different characters  $\chi=(\alpha,\beta)$ such that both $H^0(\omega^{\otimes 2}_{C})^{(\chi)}$ and $H^0(\omega_{C})^{(-\chi')}$ do not vanish, we can pick one general element in each summand $H^0(\omega^{\otimes 2}_{C})^{(\chi)} \otimes H^0(\omega_{C})^{(-\chi')}$ to form a set of six linearly independent vectors of $H^0(\Omega^1_{S_n} \otimes \Omega^2_{S_n})$. Computing the six Kas maps corresponding to the nodes on them we can check condition 2 in Theorem \ref{howtodoit}  for $X_n$. 

We will see that a reasonable choice for the characters is the following: 
\begin{definition}
We take as set of ``good" characters $\mathcal{C} :=\{\chi_0, \tilde{\chi}_0, \chi_{\infty}, \tilde{\chi}_\infty, \chi_1, \tilde{\chi}_1\}$, where
\begin{align*}
\chi_0&=(2,n-2)&\chi_\infty&=(0,2)&\chi_1&=(n-2,0)\\
\tilde{\chi}_0&=\left(2, \frac{n}{2}-2\right)&\tilde{\chi}_\infty&=\left( \frac{n}{2},2\right)&\tilde{\chi}_1&=\left( \frac{n}{2}-2, \frac{n}{2}\right).
\end{align*}
\end{definition}
\begin{remark}
By (\ref{twist}) we get that
\begin{align*}
-\chi'_0&=(2,2)&-\chi'_\infty&=(n-4,2)&-\chi'_1&=(2,n-4)\\
-\tilde{\chi}'_0&=\left(2, \frac{n}{2}+2\right)&-\tilde{\chi}'_\infty&=\left( \frac{n}{2}-4,2\right)&-\tilde{\chi}'_1&=\left( \frac{n}{2}+2, \frac{n}{2}-4\right).
\end{align*}
\end{remark}

Looking at  Figures  \ref{fig:omega1} and \ref{fig:omega2} we notice that for all $\chi \in \mathcal{C}$ we have that 
$$H^0(\omega_C^{\otimes 2}) ^{(\chi)} \otimes H^0(\omega_C) ^{(-\chi')} \neq 0.$$
We then take general elements
\begin{align*}
\omega_p &\in H^0(\omega_C^{\otimes 2}) ^{(\chi_p)} \otimes H^0(\omega_C) ^{(-\chi'_p)}, &
\tilde{\omega}_p&\in H^0(\omega_C^{\otimes 2}) ^{(\tilde{\chi}_p)} \otimes H^0(\omega_C) ^{(-\tilde{\chi}'_p)},
\end{align*}
and we want to compute the Kas map of  each of the six nodes evaluated at them. This is done  pulling  these forms back to a local double cover of the node, then choose local coordinates in it, and use (\ref{kas}).

For each $p \in \{0,1, \infty \}$, we choose $q_p \in C$ lying over $p$. Then a neighbourhood of $(q_p, q_p) $ in $C \times C$ provides a local double cover  of a node $\nu_p$  of $X_n$, allowing us to compute the Kas maps in three of the six nodes of $X_n$.  
For this we need to choose suitable local coordinates in $(q_p,q_p)$.
\begin{remark} \label{chidig}
Observe that by \cite{Par91} the natural isomorphism
$$H^0((p_* \omega_C ^{\otimes k})^{\chi}) \ra H^0(\omega_C^{\otimes k})^{\chi}$$
is given by $\sigma \mapsto p^*\sigma \otimes \rho_k^{\chi}$, where the divisor of $\rho_k^{\chi}$ is given by 
$$(\rho_k^{\chi}) = \sum_{q \in R} a_k^{\chi}(q) q,$$
where $R=  \sum R_{H,\psi}$ is the ramification divisor of $p \colon C \ra \PP^1$, and $0 \leq a_k^{\chi}(q) \leq |H|-1$. 
In fact, as in the proof of Proposition \ref{bicanonicalsplitting}, if $z$ is a local coordinate near $q \in R_{H,\psi}$ such that for each $h \in H$ it holds $h^*z = \psi(h)z$, then a global section $\omega \in H^0(\omega_C^{\otimes k})^{\chi}$ is locally of the form 
$$\lambda z^{a_k^{\chi}(q)}(dz)^k \mod z^ {|H|}.$$
$a_k^{\chi}(q)$ is then determined by the formula 
$$
\chi(h) = \psi(h)^{a_k^{\chi}(q) + k}, \ \  \forall h \in H.
$$
Note moreover that,  if $H^0(\omega_C^{\otimes k})^{\chi}\neq 0$, a general $\omega$ in it has $\lambda \neq 0$, since otherwise $p_* \omega_C^{\otimes k}$ would have a base point, which is impossible for a line bundle on $\PP^1$.

\end{remark}

Doubling the above local coordinate $z$, i.e., choosing such a local coordinate for each factor of $C \times C$ we obtain coordinates $(x_p,y_p)$ near $(q_p,q_p) \in C \times C$ such that  for all $\chi \in {\mathcal C}$ a  general
\[
\omega \in H^0(\omega_C^{\otimes 2}) ^{(\chi)} \otimes H^0(\omega_C) ^{(-\chi')}
\]
has the form $\lambda x_p^{a_2^{\chi}(q)}y_p^{a_1^{-\chi'}(q)}(dx_p)^2 \otimes dy_p \mod (x_p^n,y_p^n)$.

By (\ref{kas}) it follows that $\alpha_{\nu_p}(\omega) \neq 0$ if and only if $a_1^{-\chi'}(q_p)=1$ and $a_2^{\chi}(q_p)=0$.

\begin{proposition}\label{ThreeKasMaps}
\begin{align*}
\alpha_{\nu_p}(\omega_{p'}) \neq 0 &\Leftrightarrow p=p',&
\alpha_{\nu_p}(\tilde{\omega}_{p'}) \neq 0 &\Leftrightarrow p=p'.
\end{align*}
\end{proposition}
\begin{proof}
Setting $\chi = (\alpha, \beta)$ we get: 
\begin{align*}
\chi(g_0)  =& \alpha = a_k^{\chi}(q_0)+k && \iff & a_k^{\chi}(q_0) &= \alpha -k  &\mod& n,\\
\chi(g_\infty)  =& \beta = a_k^{\chi}(q_\infty)+k && \iff & a_k^{\chi}(q_\infty) &= \beta -k  &\mod& n,\\
\chi(g_1)  =& -\alpha-\beta = a_k^{\chi}(q_1)+k && \iff & a_k^{\chi}(q_1) &= -\alpha-\beta -k  &\mod& n.\\
\end{align*}
We have just seen that for all $\chi \in {\mathcal C}$,  a  general
$
\omega \in H^0(\omega_C^{\otimes 2}) ^{(\chi)} \otimes H^0(\omega_C) ^{(-\chi')}
$ is not in the kernel of $\alpha_{\nu_p}$ if and only if $a_1^{-\chi'}(q_p)=1$ and $a_2^{\chi}(q_p)=0$.

We solve then the system $a_1^{-\chi'}(q_p)-1=a_2^{\chi}(q_p)=0$ for all $p \in \{0,1, \infty \}$. For this we recall first that by (\ref{twist}) $-\chi' = (\alpha', \beta')$
 with
$$
\alpha' = - \alpha -2 \beta, \ \ \beta' = 2 \alpha + \beta.
$$
Then we get :

$\underline{p=0}:$ the system is 
$$
 \alpha = \alpha' = 2 \Rightarrow 2 =  -2 - 2 \beta \Rightarrow 2\beta = -4, \beta' = \beta+4,
$$
and it has exactly two solutions, $\chi_0$ and $\tilde{\chi}_0$.

$\underline{p=\infty}:$ the system is 
$$
 \beta =\beta' = 2  \Rightarrow 2 = 2 \alpha +2 \Rightarrow 2 \alpha = 0, \alpha' = -\alpha-4,
$$
and it has exactly two solutions, $\chi_\infty$ and $\tilde{\chi}_\infty$.

$\underline{p=1}:$ the system is 
$$
 -\alpha-\beta =-\alpha'-\beta'  = 2 \Rightarrow 2\beta = 0, \alpha = -\beta -2,  \beta'= \beta -4, \alpha' = \beta+2,
$$
and it has exactly two solutions, $\chi_1$ and $\tilde{\chi}_1$.

Now the claim follows immediately.
\end{proof}

Next we want to consider the other three nodes of $X_n$.

\begin{remark}\label{differentnode}
Let be $g, h \in G$, then $x'_p:=(g^{-1})^*x_p$ is a coordinate around $gq_p$ and $y_p':=(h^{-1})^*y_p$ is a coordinate around $hq_p$.  

If $\omega \in H^0(\omega_C^{\otimes 2}) ^{(\chi)} \otimes H^0(\omega_C) ^{(-\chi')} $, then 
$
(g^{-1}, h^{-1})^* \omega = \chi(g^{-1}) (\chi')^{-1}(h^{-1}) \omega
$. Writing locally near $(q_p,q_p)$
$$
\omega=f(x_p,y_p)(dx_p)^2\otimes dy_p,
$$
it follows that
$$
\omega= \frac{\chi(g)}{\chi'(h)}(g^{-1}, h^{-1})^* \omega= 
 \frac{\chi(g)}{\chi'(h)} f(x'_p,y'_p)(dx'_p)^2\otimes dy'_p.
$$
This shows that if we have computed $\alpha_{\nu_p}(\omega) = \lambda$ by (\ref{kas}) using the coordinates $(x_p,y_p)$, and if $\nu'$ is the node dominated by $(gq_p,hq_p)$ we obtain that $\alpha_{\nu'}(\omega)$ (computed by (\ref{kas}) using the coordinates $(x'_p,y'_p)$) equals   
$$\frac{\chi(g)}{\chi'(h)}\cdot \lambda.$$
\end{remark}

It is worth noticing that there are elements $(g,h)$ such that $(gq_p,hq_p)$ dominates $\nu_p$ and $\chi(g)\neq \chi'(h)$. In this case we get a different value of $\alpha_{\nu_p}(\omega)$ by applying  (\ref{kas}) with the coordinates $(x'_p,y'_p)$ instead of $(x_p,y_p)$. This is not a contradiction, as it corresponds to a change of the choice of $\theta_i$: in fact in that case the value of the Kas map in all possible forms $\omega$ will be rescaled by the same factor. This will be no more true when $(gq_p,hq_p)$  dominates a different node, as we will see in the proof of Proposition \ref{surjective}.

We are now ready to finish the proof of Theorem \ref{main}.
\begin{proposition} \label{surjective}
Let $V:=\left\langle \omega_0,\tilde{\omega}_0, \omega_\infty,\tilde{\omega}_\infty, \omega_1,\tilde{\omega}_1\right\rangle$.
The restriction to $V$ of the Kas maps of the six nodes of $X_n$ are linearly independent.

In particular, condition 2) of Theorem \ref{howtodoit} holds.
\end{proposition}

\begin{proof} First we compute the Kas maps of the three nodes $\nu_p$. 

By Proposition \ref{ThreeKasMaps}, we know that $\alpha_{\nu_p}(\omega_p)\neq 0$. Rescaling $\omega_p$ we can assume  without loss of generality  $\alpha_{\nu_p}(\omega_p)=1$. Similarly we can assume $\alpha_{\nu_p}(\tilde{\omega}_p)=1$
and then Proposition \ref{ThreeKasMaps} gives
\begin{align*}
 (\alpha_{\nu_0})(\omega_0,\tilde{\omega}_0,\omega_1,\tilde{\omega}_1,\omega_{\infty},\tilde{\omega}_{\infty})=&(1,1,0,0,0,0),\\
 (\alpha_{\nu_1})(\omega_0,\tilde{\omega}_0,\omega_1,\tilde{\omega}_1,\omega_{\infty},\tilde{\omega}_{\infty})=&(0,0,1,1,0,0),\\
(\alpha_{\nu_\infty})(\omega_0,\tilde{\omega}_0,\omega_1,\tilde{\omega}_1,\omega_{\infty},\tilde{\omega}_{\infty})=&(0,0,0,0,1,1).
  \end{align*}
We fix now $k_0=(1,1)$,  $k_{\infty}=(1,0)$, and $k_1 =\ (0,1) \in G$ and observe that for all $p \in \{0,1, \infty\}$ we have $\chi_p(k_p) = 1$, $\tilde{\chi}_p (k_p)= -1$. (Note that here the characters are written multiplicatively.)

Denoting $k_pq_p$ by $\tilde{q_p}$, we have $(\tilde{q_p}, q_p) = (k_p,0)(q_p,q_p)$. Denoting by $\tilde{\nu}_p$ the node in $X_n$ dominated by 
$(\tilde{q_p}, q_p)$, we see that in the coordinates obtained by pullback via $(k_p,0)$ from $(q_p,q_p)$ we have
$$
\alpha_{\tilde{\nu}_p}(\omega_{p'}) = \delta_{pp'}, \  \ \alpha_{\tilde{\nu}_p}(\tilde{\omega}_{p'}) = - \delta_{pp'}
$$
where $\delta_{pp'}$  is the usual Kronecker delta equal to $1$ if $p=p'$ and $0$ if $p\neq p'$.
  
This implies first of all that $\nu_p \neq \tilde{\nu}_p$ since the values of the Kas map of  $\nu_p$  in $\omega_p$ and $\tilde{\omega}_p$ coincide. Moreover, it follows  that 
 \begin{align*}
 (\alpha_{\tilde{\nu}_0})(\omega_0,\tilde{\omega}_0,\omega_1,\tilde{\omega}_1,\omega_{\infty},\tilde{\omega}_{\infty})=&(1,-1,0,0,0,0),\\
 (\alpha_{\tilde{\nu}_1})(\omega_0,\tilde{\omega}_0,\omega_1,\tilde{\omega}_1,\omega_{\infty},\tilde{\omega}_{\infty})=&(0,0,1,-1,0,0),\\
 (\alpha_{\tilde{\nu}_\infty})(\omega_0,\tilde{\omega}_0,\omega_1,\tilde{\omega}_1,\omega_{\infty},\tilde{\omega}_{\infty})=&(0,0,0,0,1,-1).
  \end{align*}

  This concludes the proof.
  \end{proof}

\section{Higher dimensional examples}

The aim of this section is to give examples of rigid compact complex manifolds which are not infinitesimally rigid  in all dimensions $d \geq 3$.

The main result is the following

\begin{theorem}\label{higherdim}
Let $n \geq 8$ be an even integer such that $3 \nmid n$, and let $X$ be a compact complex rigid manifold.

Then $S_n \times X$ is rigid, but not infinitesimally rigid. 

In particular there are rigid, but not infinitesimally rigid, manifolds of dimension $d$ and Kodaira dimension $\kappa$ for all possible pairs $(d,\kappa)$ with $d \geq 5$ and $\kappa \neq 0,1,3$ and for $(d,\kappa)=(3,-\infty),(4,-\infty),(4,4)$.
\end{theorem}

Let $X$ and $Y$ be compact complex manifolds. Then by the K\"unneth formula (cf. \cite{kaup})  we have:
\begin{multline*}
H^1(\Theta_{X\times Y}) = H^1(\Theta_X) \oplus (H^0(\Theta_X) \otimes H^1(\hol_Y)) \oplus \\
\oplus (H^1(\hol_X) \otimes H^0(\Theta_Y)) \oplus H^1(\Theta_Y).
\end{multline*}

Before proving the Theorem we need the following result, which is probably wellknown. For lack of a suitable reference we will give a sketch of proof.

\begin{lemma}\label{wellknown}
Let $X$, $Y$ be compact complex manifolds, such that 
\begin{equation}\label{h0h1=0}
h^0(\Theta_X) h^1(\hol_Y)=h^0(\Theta_Y)h^1(\hol_X) =0.
\end{equation}
Then $\Def(X \times Y) = \Def(X) \times \Def(Y)$.
\end{lemma}

\begin{proof}
Consider the commutative diagram
\[
\begin{xy}
  \xymatrix{
     H^1(\Theta_X) \oplus  H^1(\Theta_Y)\ar[rr]^{\alpha_1}\ar[d]_{k_X \oplus k_Y}  &  & H^1(\Theta_{X\times Y})\ar[d]^{k_{X \times Y}} \\   
		H^2(\Theta_X) \oplus  H^2(\Theta_Y) \ar[rr]_{\alpha_2}  &   & H^2(\Theta_{X\times Y}).
		}
\end{xy}
\]
where the horizontal map $\alpha_j$ are the natural inclusions given by the  K\"unneth formula, and the maps $k_\bullet \colon H^1(\Theta_\bullet) \rightarrow  H^2(\Theta_\bullet)$ are the obstruction maps defining $\Def(\bullet):=k_\bullet^{-1}(0)$.

Condition (\ref{h0h1=0}) means that $\alpha_1$ is an isomorphism. Since $\alpha_2$ is injective, $\alpha_1$ maps $ \Def (X) \times \Def (Y)= (k_X \oplus k_Y)^{-1}(0)$ isomorphically onto $\Def (X \times Y)=k_{X \times Y}^{-1}(0)$.
\end{proof}

\begin{proof}[Proof of Theorem \ref{higherdim}]
Observe that $H^1(\hol_{S_n}) = H^0(\Theta_{S_n}) = 0$. Therefore applying the K\"unneth formula we obtain:
$$
H^1(\Theta_{S_n \times X}) = H^1(\Theta_{S_n}) \oplus H^1(\Theta_X) \neq \{0\}.
$$

Then $S_n \times X$ is not infinitesimally rigid. Moreover, by the above Lemma, we have that
$$
\Def(S_n \times X) = \Def(S_n) \times \Def(X).
$$
Therefore, since $S_n$ and $X$ are rigid, also $S_n \times X$ is rigid.

Choosing $X=(\PP^1)^{d-2}$ we get examples for all dimension $d \geq 3$ with $\kappa=-\infty$.

Choosing $X=S_m$ ($m \geq 8$ even with $3 \nmid m$) we get  examples with $(d,\kappa) =(4,4)$.

Choosing a rigid manifold $X$ of Kodaira dimension $\kappa \in \{0, 2, \ldots, \dim X \}$ (cf. \cite{rigidity}*{Theorems 3.4, 3.5}), we get rigid and not infinitesimally rigid examples for all dimensions $d \geq 5$ and all possible Kodaira dimensions except $0,1,3$.
\end{proof}

\begin{remark}
Observe that, since rigid manifolds of general type are globally rigid, we found globally rigid manifolds which are not infinitesimally rigid of every dimension $d \geq 2$, $d \neq 3$.
\end{remark}

\end{document}